\documentclass[11pt,a4,fleqn]{article}
\usepackage{graphicx}
\usepackage{amsmath,amssymb,latexsym,graphics,epsfig}
\usepackage{hyperref}
\usepackage{color}
\usepackage{amsthm}

\newtheorem{theorem}{\bf Theorem}[section]

\newtheorem{lemma}[theorem]{\bf Lemma}
\newtheorem{corollary}[theorem]{\bf Corollary}

\newtheorem{conjecture}{\bf Conjecture}

\numberwithin{equation}{section}

\begin{document}
\title{{\Large Ramsey numbers of $4$-uniform loose  cycles}}
\author{\small  G.R. Omidi$^{\textrm{a},\textrm{b},1}$, M. Shahsiah$^{\textrm{b}}$\\
\small  $^{\textrm{a}}$Department of Mathematical Sciences,
Isfahan University
of Technology,\\ \small Isfahan, 84156-83111, Iran\\
\small  $^{\textrm{b}}$School of Mathematics, Institute for
Research
in Fundamental Sciences (IPM),\\
\small  P.O.Box: 19395-5746, Tehran,
Iran\\
\small \texttt{E-mails: romidi@cc.iut.ac.ir,
m.shahsiah@math.iut.ac.ir}}
\date {}
\maketitle \footnotetext[1] {\tt This research is partially
carried out in the IPM-Isfahan Branch and in part supported
by a grant from IPM (No. 92050217).} \vspace*{-0.5cm}

\begin{abstract}Gy\'{a}rf\'{a}s, S\'{a}rk\"{o}zy and
Szemer\'{e}di proved that the $2$-color Ramsey number
$R(\mathcal{C}^k_n,\mathcal{C}^k_n)$ of a $k$-uniform loose cycle
$\mathcal{C}^k_n$ is asymptotically $\frac{1}{2}(2k-1)n,$  generating the same result for $k=3$ due to Haxell et al. Concerning their results, it is
 conjectured that
 for every $n\geq m\geq 3$ and $k\geq 3,$
$$R(\mathcal{C}^k_n,\mathcal{C}^k_m)=(k-1)n+\lfloor\frac{m-1}{2}\rfloor.$$
In $2014$, the case $k=3$  is  proved by the
authors. Recently, the authors showed that this conjecture is true for $n=m\geq 2$ and $k\geq 8$. Their method can be used for case $n=m\geq 2$ and $k=7,$ but more details are required. The only open cases for the above conjecture when $n=m$ are $k=4,5,6.$ Here we investigate to the case $k=4$
and  we show that  the conjecture holds
for $k=4$  when $n>m$ or $n=m$ is odd. When $n=m$ is even, we show that
 $R(\mathcal{C}^4_n,\mathcal{C}^4_n)$ is between two values with difference one.


\end{abstract}


\section{ Introduction}

For given $k$-uniform hypergraphs  $\mathcal{G}$ and
$\mathcal{H},$
  the \textit{Ramsey number} $R(\mathcal{G},\mathcal{H})$ is
  the smallest positive integer $N$ such that in every red-blue coloring of the
  edges of the complete $k$-uniform hypergraph $\mathcal{K}^k_N$,  there is  a red copy of $\mathcal{G}$ or a blue copy of $\mathcal{H}.$
 A {\it $k$-uniform  loose cycle} $\mathcal{C}_n^k$ (shortly, a
{\it cycle of length $n$})  is a hypergraph with vertex set
$\{v_1,v_2,\ldots,v_{n(k-1)}\}$ and with the set of $n$ edges
$e_i=\{v_{(i-1)(k-1)+1},v_{(i-1)(k-1)+2},\ldots, v_{(i-1)(k-1)+k}\}$, $1\leq i\leq n$, where
we use mod $n(k-1)$ arithmetic.
Similarly, a
{\it $k$-uniform  loose path} $\mathcal{P}_n^k$ (shortly, a {\it
path of length $n$}) is a hypergraph with vertex set
$\{v_1,v_2,\ldots,v_{n(k-1)+1}\}$  and with the set of $n$ edges
$e_i=\{v_{(i-1)(k-1)+1},v_{(i-1)(k-1)+2},\ldots, v_{(i-1)(k-1)+k}\}$, $1\leq i\leq n$. For
an edge $e_i=\{v_{(i-1)(k-1)+1},v_{(i-1)(k-1)+2},\ldots, v_{i(k-1)+1}\}$ of a given loose path (also a given loose cycle)
$\mathcal{K}$, the first vertex ($v_{(i-1)(k-1)+1}$) and the last vertex ($v_{i(k-1)+1}$) are denoted by
$f_{\mathcal{K},e_i}$ and
$l_{\mathcal{K},e_i}$, respectively. In
this paper, we consider the problem of finding the 2-color Ramsey
number of $4$-uniform loose paths and cycles.\\

The investigation of the Ramsey numbers of hypergraph loose
cycles was initiated by  Haxell et al. in \cite{Ramsy number of
loose cycle}. They proved that
$R(\mathcal{C}^3_n,\mathcal{C}^3_n)$ is asymptotically
$\frac{5}{2}n$.
This result was extended by Gy\'{a}rf\'{a}s, S\'{a}rk\"{o}zy and
Szemer\'{e}di \cite{Ramsy number of loose cycle for k-uniform} to
$k$-uniform loose cycles. More precisely, they proved that for all
$\eta>0$ there exists $n_0=n_0(\eta)$ such that for every $n> n_0,$ every 2-coloring
of $\mathcal{K}^{k}_N$ with $N=(1+\eta)\frac{1}{2}(2k-1)n$
contains a monochromatic copy of $\mathcal{C}^k_n.$\\

 In
\cite{subm}, Gy\'{a}rf\'{a}s and Raeisi  determined the value of
the Ramsey number of a $k$-uniform loose triangle and quadrangle.
 Recently, we proved  the following general result on the Ramsey
numbers of loose paths and loose cycles in 3-uniform hypergraphs.

\begin{theorem}{\rm \cite{The Ramsey number of loose paths  and loose cycles
in 3-uniform hypergraphs}}\label{Omidi} For every $n\geq m\geq 3,$
\begin{eqnarray*}\label{5}
R(\mathcal{P}^3_n,\mathcal{P}^3_m)=R(\mathcal{P}^3_n,\mathcal{C}^3_m)=R(\mathcal{C}^3_n,\mathcal{C}^3_m)+1=2n+\Big\lfloor\frac{m+1}{2}\Big\rfloor.
\end{eqnarray*}
\end{theorem}

In \cite{Ramsey numbers of loose cycles in uniform hypergraphs},
we presented another proof of Theorem \ref{Omidi} and posed the
following conjecture.

\begin{conjecture}\label{our conjecture}
Let $k\geq 3$ be an integer number. For every $n\geq m \geq 3$,
\begin{eqnarray*}\label{2}
R(\mathcal{P}^k_{n},\mathcal{P}^k_{m})=R(\mathcal{P}^k_{n},\mathcal{C}^k_{m})=R(\mathcal{C}^k_n,\mathcal{C}^k_m)+1=(k-1)n+\lfloor\frac{m+1}{2}\rfloor.
\end{eqnarray*}
\end{conjecture}

\noindent Also,  the following theorem is obtained on the Ramsey number of loose paths and cycles in $k$-uniform hypergraphs  \cite{Ramsey numbers of loose cycles in uniform
hypergraphs}.

\begin{theorem}{\rm \cite{Ramsey numbers of loose cycles in uniform hypergraphs}}\label{connection}
Let $n\geq m\geq 2$ be given integers and
$R(\mathcal{C}^k_{n},\mathcal{C}^k_{m})=(k-1)n+\lfloor\frac{m-1}{2}\rfloor.$
Then
$R(\mathcal{P}^k_n,\mathcal{C}^k_m)=(k-1)n+\lfloor\frac{m+1}{2}\rfloor$
and
$R(\mathcal{P}^k_n,\mathcal{P}^k_{m-1})=(k-1)n+\lfloor\frac{m}{2}\rfloor$.
Moreover, for $n=m$ we have
$R(\mathcal{P}^k_n,\mathcal{P}^k_m)=(k-1)n+\lfloor\frac{m+1}{2}\rfloor$.
\end{theorem}

\noindent  Using Theorem \ref{connection}, one can easily see that   Conjecture \ref{our conjecture} is equivalent to the following.

\begin{conjecture}\label{our conjecture2}
Let $k\geq 3$ be an integer number. For every $n\geq m \geq 3$,
\begin{eqnarray*}\label{2}
R(\mathcal{C}^k_n,\mathcal{C}^k_m)=(k-1)n+\lfloor\frac{m-1}{2}\rfloor.
\end{eqnarray*}
\end{conjecture}

Recently, it is shown that Conjecture \ref{our conjecture2} holds for $n=m$ and $k\geq 8$ (see \cite{diagonal}). As we mentioned in \cite{diagonal}, our methods can be used to prove Conjecture \ref{our conjecture2} for $n=m$ and $k\geq 7.$ Therefore, based on Theorem \ref{Omidi}, the cases $k=4,5,6$ are the only open cases for Conjecture \ref{our conjecture2} when $n=m$ (the problem of determines the diagonal Ramsey number of loose cycles). In this paper, we investigate Conjecture \ref{our conjecture2} for  $k=4$. More precisely, we
 extend the method that used in  \cite{Ramsey numbers of
loose cycles in uniform hypergraphs} and we  show that
Conjecture \ref{our conjecture2}
holds for $k=4$ where
$n>m$ or $n=m$ is odd. When $n=m$ is even we show that $R(\mathcal{C}^4_n,\mathcal{C}^4_m)$  either is the value that is claimed in Conjecture \ref{our conjecture2} or is equal to this value minus one.
  Consequently, using Theorem
\ref{connection}, we obtained   the values of some Ramsey numbers
involving paths. Throughout the paper, by Lemma 1 of
\cite{subm}, it  suffices to
 prove only the  upper bound for the claimed Ramsey numbers.
Throughout the paper, for
a 2-edge colored hypergraph $\mathcal{H}$ we denote by
$\mathcal{H}_{\rm red}$ and $\mathcal{H}_{\rm blue}$ the induced
hypergraphs on red edges and blue edges, respectively. Also we
denote by $|\mathcal{H}|$ and $\|\mathcal{H}\|$ the number of
vertices and edges of $\mathcal{H}$, respectively.


\section{ Preliminaries}

In this section, we prove  some  lemmas that will be needed in our main results.
 Also, we recall some results from \cite{subm} and \cite{Ramsey numbers of loose cycles in uniform hypergraphs}.

 \begin{theorem}\label{R(Pk3,Pk3)} {\rm \cite{subm}}
For every $k\geq 3$,
\begin{itemize}
\item [\rm{(a)}]~
$R(\mathcal{P}^k_3,\mathcal{P}^k_3)=R(\mathcal{C}^k_3,\mathcal{P}^k_3)=R(\mathcal{C}^k_3,\mathcal{C}^k_3)+1=3k-1$,
\item [\rm{(b)}]~
$R(\mathcal{P}^k_4,\mathcal{P}^k_4)=R(\mathcal{C}^k_4,\mathcal{P}^k_4)=R(\mathcal{C}^k_4,\mathcal{C}^k_4)+1=4k-2$.
\end{itemize}
\end{theorem}

 \begin{theorem}\label{R(C3,C4)}{\rm \cite{Ramsey numbers of loose cycles in uniform hypergraphs}}
Let $n,k\geq 3$ be  integer numbers. Then
\begin{eqnarray*}R(\mathcal{C}^k_3,\mathcal{C}^k_n)= (k-1)n+1.\end{eqnarray*}
\end{theorem}


\vspace{0.5 cm}

In order to state our main results we need some definitions. Let
$\mathcal{H}$ be a 2-edge colored complete $4$-uniform hypergraph,
$\mathcal{P}$ be a loose path in $\mathcal{H}$ and $W$ be a set of
vertices with $W\cap V(\mathcal{P})=\emptyset$. By a {\it
$\varpi_S$-configuration}, we mean a copy of $\mathcal{P}^4_2$
with edges $$\{x,a_1,a_2,a_3\},
\{a_3,a_4,a_5,y\},$$ so that $\{x,y\}\subseteq W$ and  $S=\{a_j : 1\leq j
\leq 5\}\subseteq (e_{i-1}\setminus \{f_{\mathcal{P},e_{i-1}}\})\cup e_{i}
\cup e_{i+1}$ is a set of unordered vertices of
 $3$
 consecutive edges of
$\mathcal{P}$ with  $|S\cap (e_{i-1}\setminus \{f_{\mathcal{P},e_{i-1}}\})|\leq 1.$
 The vertices $x$ and $y$ are called {\it
the end vertices} of this configuration. A
$\varpi_{S}$-configuration, $S\subseteq (e_{i-1}\setminus \{f_{\mathcal{P},e_{i-1}}\})\cup e_{i}
\cup e_{i+1}$,
 is {\it good} if at least one of the vertices of
$e_{i+1}\setminus e_{i}$
is not in $S$. We say that a monochromatic path
$\mathcal{P}=e_1e_2\ldots e_n$ is {\it maximal with respect to}
(w.r.t. for short) $W\subseteq V(\mathcal{H})\setminus
V(\mathcal{P})$ if there is no $W'\subseteq W$ so that for some
$1\leq r\leq n$ and $1\leq i\leq n-r+1,$
 \begin{eqnarray*}\mathcal{P}'=e_1e_2\ldots e_{i-1}e'_ie'_{i+1}\ldots e'_{i+r}e_{i+r}\ldots e_n,\end{eqnarray*} is a monochromatic path with $n+1$ edges and
 the following properties:

 \begin{itemize}
\item[(i)] $V(\mathcal{P}')=V(\mathcal{P})\cup W'$,
 \item[(ii)] if $i=1$, then
$f_{\mathcal{P}',e'_i}=f_{\mathcal{P},e_i}$,
\item[(iii)] if
 $i+r-1=n$, then
$l_{\mathcal{P}',e'_{i+r}}=l_{\mathcal{P},e_n}$.
\end{itemize}
 Clearly, if $\mathcal{P}$ is maximal w.r.t. $W$, then it is maximal w.r.t. every
$W'\subseteq W$ and also every loose path $\mathcal{P}'$ which is a sub-hypergraph of $\mathcal{P}$ is again maximal w.r.t. $W$.

   \noindent   We use these definitions to deduce the following  essential lemma.

\bigskip
\begin{lemma}\label{spacial configuration1}
Assume that  $\mathcal{H}=\mathcal{K}^4_{n}$  is
$2$-edge colored red and blue. Let $\mathcal{P}\subseteq
\mathcal{H}_{\rm red}$ be a maximal path w.r.t. $W,$ where
$W\subseteq V(\mathcal{H})\setminus V(\mathcal{P})$ and  $|W|\geq 4$.
For every two consecutive edges $e_1$ and $e_2$ of
$\mathcal{P}$ there is a
 good $\varpi_S$-configuration, say
$C=fg$, in $\mathcal{H}_{\rm blue}$ with  end vertices $x\in f$ and $y\in g$ in $W$ and
$S\subseteq e_1\cup e_{2}$. Moreover, there are two subsets $W_1\subseteq W$ and $W_2\subseteq W$ with $|W_1|\geq |W|-2$ and $|W_2|\geq |W|-3$ so that for every distinct vertices $x'\in W_1$ and $y'\in W_2$, the path $C'=\Big((f\setminus\{x\})\cup\{x'\}\Big)\Big((g\setminus\{y\})\cup\{y'\}\Big)$ is also a good $\varpi_S$-configuration in $\mathcal{H}_{\rm blue}$ with  end vertices $x'$ and $y'$ in $W.$

\end{lemma}

\begin{proof}{ Let
$$e_1=\{v_1,v_2,v_3,v_4\},  e_2=\{v_4,v_5,v_6,v_7\}.$$

\noindent Among  different choices of $3$ distinct vertices of $W,$ choose a $3$-tuple
$X=(x_1,x_2,x_3)$ so that $E_X$ has the  minimum
number of  blue edges, where $E_X=\{f_1,f_2,f_3\}$ and
\begin{eqnarray*}
&&f_{1}=\{v_1,x_1,v_2,v_5\},\\
&&f_{2}=\{v_2,x_2,v_3,v_6\},\\
&&f_{{3}}=\{v_3,x_{3},v_4,v_7\}.
\end{eqnarray*}
Note that for $1\leq i \leq 3,$ we have $|f_i\cap(e_{2}\setminus\{f_{\mathcal{P},e_{2}}\})|=1.$
 Since $\mathcal{P}$ is a maximal path w.r.t.   $W,$ there is $1\leq j\leq
3$ so that  the edge $f_{j}$ is blue. Otherwise, replacing $e_1e_2$ by $f_1f_2f_3$ in $\mathcal{P}$ yields a red path $\mathcal{P}'$ with $n+1$ edges; this is a contradiction.
 Let $W_1=(W\setminus\{x_1,x_2,x_{3}\})\cup\{x_j\}$.
For each vertex
$x\in W_1$  the edge
$f_x=(f_j\setminus\{x_j\})\cup\{x\}$ is blue. Otherwise, the number of blue edges in $E_Y$
 is less than this number for $E_X$, where $Y$ is obtained from $X$  by replacing  $x_j$ to $x$. This is a contradiction.\\

Now we choose $h_1,h_2,h_3$ as follows. If $j=1,$ then set
\begin{eqnarray*}
h_1=\{v_1,v_6,v_3\},h_2=\{v_3,v_2,v_4\},h_3=\{v_4,v_5,v_7\}.
\end{eqnarray*}
If $j=2,$ then set
\begin{eqnarray*}
h_1=\{v_1,v_2,v_5\},h_2=\{v_5,v_6,v_4\},h_3=\{v_4,v_3,v_7\}.
\end{eqnarray*}
If $j=3,$ then set
\begin{eqnarray*}
h_1=\{v_1,v_3,v_5\},h_2=\{v_5,v_4,v_2\},h_3=\{v_2,v_6,v_7\}.
\end{eqnarray*}
Note that in each the above cases, for $1\leq i \leq 3,$ we have  $|h_i\cap (f_j\setminus\{x_j\})|=1$ and $|h_i\cap(e_2\setminus(f_j\cup\{v_4\}))|\leq 1$.
Let $Y=(y_1,y_2,y_{3})$ be a $3$-tuple of distinct vertices of $W\setminus\{x_j\}$ with minimum number of blue edges in $F_{Y}$,  where $F_Y=\{g_1,g_2,g_3\}$ and $g_i=h_i\cup\{y_i\}$.
 Again since $\mathcal{P}$ is maximal w.r.t. $W$, for some $1\leq \ell\leq 3$ the edge
$g_{\ell}$ is blue and also, for each vertex $y_a\in W_2=
(W\setminus\{x_j,y_{1},y_{2},y_3\})\cup \{y_\ell\}$ the edge $g_a=(g_{\ell}\setminus\{y_{\ell}\})\cup\{y_a\}$ is blue. Therefore, for every $x'\in W_1$ and $y'\in W_2$,
 we have $C=fg$ which is our desired configuration, where $f=(f_j\setminus\{x_j\})\cup\{x'\}$ and $g=(g_{\ell}\setminus\{y_{\ell}\})\cup\{y'\}.$
  Since $|W_1|=|W|-2$, each vertex of $W,$ with the exception of at most $2,$ can be considered as an end vertex of $C.$
  Note that this
configuration contains at most two vertices of $e_2\setminus
e_1$.}\end{proof}

By an argument similar to the proof of Lemma
 \ref{spacial configuration1}, we have the following general result.

\bigskip
\begin{lemma}\label{spacial configuration2}
Assume that   $\mathcal{H}=\mathcal{K}^4_{n},$  is
$2$-edge colored red and blue. Let $\mathcal{P}\subseteq \mathcal{H}_{\rm red}$ be a maximal path w.r.t. $W,$
where $W\subseteq V(\mathcal{H})\setminus V(\mathcal{P})$ and  $|W|\geq 4$.
Let $A_1=\{f_{\mathcal{P},e_{1}}\}=\{v_1\}$ and
$A_i=V(e_{i-1})\setminus\{f_{\mathcal{P},e_{i-1}}\}$ for $i>1$.
 Then for every two  consecutive edges $e_i$ and $e_{i+1}$
of $\mathcal{P}$ and for each $u\in A_i$ there  is a
 good $\varpi_S$-configuration, say
$C=fg$, in $\mathcal{H}_{\rm blue}$ with  end vertices  $x\in f$ and $y\in g$  in $W$ and
 \begin{eqnarray*}S\subseteq \Big((e_i\setminus
\{f_{\mathcal{P},e_{i}}\})\cup \{u\}\Big)\cup \Big(e_{i+1}\setminus \{v\}\Big),\end{eqnarray*} for
some $v\in A_{i+2}$. Moreover, there are two subsets $W_1\subseteq W$ and $W_2\subseteq W$ with $|W_1|\geq |W|-2$ and $|W_2|\geq |W|-3$ so that for every distinct vertices $x'\in W_1$ and $y'\in W_2$, the path $C'=\Big((f\setminus\{x\})\cup\{x'\}\Big)\Big((g\setminus\{y\})\cup\{y'\}\Big)$ is also a good $\varpi_S$-configuration in $\mathcal{H}_{\rm blue}$ with  end vertices $x'$ and $y'$ in $W.$


\end{lemma}

The following result is an immediate corollary of Lemma
 \ref{spacial configuration2}.
\begin{corollary}\label{there is a Pl}
Let   $\mathcal{H}=\mathcal{K}_l^4$  be two edge
colored red and blue. Also let $\mathcal{P}=e_1e_2\ldots e_n,$ $n\geq
2,$ be a maximal red path w.r.t. $W,$ where $W\subseteq
V(\mathcal{H})\setminus V(\mathcal{P})$ and  $|W|\geq 4$. Then for
some $r\geq 0$ and $W'\subseteq W$ there are  two disjoint blue
paths $\mathcal{Q}$ and $\mathcal{Q}',$ with $\|\mathcal{Q}\|\geq 2$ and
\begin{eqnarray*}
\|\mathcal{Q}\cup \mathcal{Q}'\|=n-r=\left\lbrace
\begin{array}{ll}
2(|W'|-2)  &\mbox{if} \ \|\mathcal{Q}'\|\neq 0,\vspace{.5 cm}\\
2(|W'|-1) & \mbox{if}\ \|\mathcal{Q}'\|=0,
\end{array}
\right.\vspace{.2 cm}
\end{eqnarray*}
 between $W'$ and
$\overline{\mathcal{P}}=e_1e_2\ldots e_{n-r}$ so that $e\cap W'$
is actually the end vertex of $e$ for each edge $e\in
\mathcal{Q}\cup \mathcal{Q}'$ and at least one of the vertices of
$e_{n-r}\setminus e_{n-r-1}$ is not in $V(\mathcal{Q})\cup
V(\mathcal{Q}')$. Moreover, if $\|\mathcal{Q}'\|=0$ then either
$x=|W\setminus W'|\in\{1,2\}$ or $x\geq 3$ and $0\leq r \leq
1$. Otherwise, either $x=|W\setminus W'|=0$ or $x\geq 1$ and
$0\leq r \leq 1$.

\end{corollary}

\begin{proof}{
Let $\mathcal{P}=e_1e_2\ldots e_{n}$ be a maximal  red path w.r.t.
$W,$ $W\subseteq V(\mathcal{H})\setminus V(\mathcal{P})$, and
\begin{eqnarray*} e_i=\{v_{(i-1)(k-1)+1},v_{(i-1)(k-1)+2},\ldots,v_{i(k-1)+1}\}, \hspace{1 cm} i=1,2,\ldots,n,\end{eqnarray*}
are the edges of $\mathcal{P}$.\\\\
{\bf Step 1:} Set  $\mathcal{P}_1=\mathcal{P}$, $W_1=W$ and
$\overline{\mathcal{P}}_1=\mathcal{P}'_1=e_{1}e_{2}$. Since $\mathcal{P}$ is
maximal w.r.t. $W_1$, using Lemma
 \ref{spacial configuration1} there is a good $\varpi_S$-configuration, say
 $\mathcal{Q}_1=f_1g_1,$ in $\mathcal{H}_{\rm blue}$  with end vertices $x\in f_1$ and $y\in g_1$ in $W_1$
 so that $S\subseteq \mathcal{P}'_1$  and
  $\mathcal{Q}_1$ does not contain  a vertex of
 $e_2\setminus e_1,$ say $u_1$. Set  $X_1=|W\setminus V(\mathcal{Q}_1)|$,
$\mathcal{P}_2=\mathcal{P}_1\setminus \overline{\mathcal{P}}_1=e_3e_4\ldots e_n$ and
$W_2=W.$
 If  $|W_2|=4$ or $\|\mathcal{P}_2\|\leq 1$, then $\mathcal{Q}=\mathcal{Q}_1$ is a blue path between $W'=W_1\cap V(\mathcal{Q}_1)$
 and $\overline{\mathcal{P}}=\overline{\mathcal{P}}_1$ with desired properties. Otherwise, go to Step
 2.\\\\

\noindent {\bf Step 2:}
 Clearly  $|W_2|\geq 5$ and $\|\mathcal{P}_2\|\geq 2.$
 Set $\overline{\mathcal{P}}_2=e_3e_4$ and
 $\mathcal{P}'_2=\Big((e_{3}\setminus \{f_{\mathcal{P},e_{3}}\})\cup\{u_1\}\Big)e_4$.
 Since $\mathcal{P}$ is maximal w.r.t. $W_2$,
 using Lemma \ref{spacial configuration2} there is  a good $\varpi_S$-configuration, say
 $\mathcal{Q}_2=f_2g_2,$ in $\mathcal{H}_{\rm blue}$ with end vertices $x\in f_2$ and $y\in g_2$ in $W_2$
 such that $S\subseteq \mathcal{P}'_2$  and  $\mathcal{Q}_2$ does not contain a vertex of $e_4\setminus e_3,$ say $u_2.$ By Lemma \ref{spacial configuration2}, there are two subsets $W_{21}\subseteq W_2$ and $W_{22}\subseteq W_2$ with $|W_{21}|\geq |W_2|-2$ and $|W_{22}|\geq |W_2|-3$ so that for every distinct vertices $x'\in W_{21}$ and $y'\in W_{22}$, the path $\mathcal{Q}'_2=\Big((f_2\setminus\{x\})\cup\{x'\}\Big)\Big((g_2\setminus\{y\})\cup\{y'\}\Big)$ is also a good $\varpi_S$-configuration in $\mathcal{H}_{\rm blue}$ with  end vertices $x'$ and $y'$ in $W_2.$
Therefore, we may assume that   $\bigcup_{i=1}^{2}\mathcal{Q}_i$ is either  a blue path or the union of  two disjoint blue
paths.
 Set $X_2=|W\setminus \bigcup_{i=1}^{2} V(\mathcal{Q}_i)|$ and
$\mathcal{P}_3=\mathcal{P}_{2}\setminus
\overline{\mathcal{P}}_{2}=e_5e_6\ldots e_n$. If
$\bigcup_{i=1}^{2}\mathcal{Q}_i$ is a blue path $\mathcal{Q}$  with end vertices $x_{2}$ and $y_{2}$,
then set
\begin{eqnarray*}
W_3=\Big(W_{2}\setminus
V(\mathcal{Q})\Big)\cup\{x_{2},y_{2}\}.
\end{eqnarray*}
In this case, clearly $|W_3|= |W_2|-1$.
Otherwise, $\bigcup_{i=1}^{2}\mathcal{Q}_i$ is the union of  two disjoint blue
paths $\mathcal{Q}$ and $\mathcal{Q}'$ with end vertices $x_{2},y_{2}$ and $x'_{2},y'_{2}$ in $W_2$, respectively. In this case,  set
\begin{eqnarray*}
W_3=\Big(W_{2}\setminus
V(\mathcal{Q}\cup\mathcal{Q}')\Big)\cup\{x_{2},y_{2},x'_{2},y'_{2}\}.
\end{eqnarray*}
Clearly $|W_3|=|W_2|$. If  $|W_3|\leq 4$ or $\|\mathcal{P}_3\|\leq 1$,
 then   $\bigcup_{i=1}^{2}\mathcal{Q}_i=\mathcal{Q}$ and  $\emptyset$  or $\mathcal{Q}$ and $\mathcal{Q}'$ (in the case $\bigcup_{i=1}^{2}\mathcal{Q}_i=\mathcal{Q}\cup\mathcal{Q}'$) are the paths between $W'=W\cap \bigcup_{i=1}^{2} V(\mathcal{Q}_i)$ and $\overline{\mathcal{P}}=\overline{\mathcal{P}}_1\cup \overline{\mathcal{P}}_2$
  with desired properties. Otherwise, go to Step $3$.\\\\\\

\noindent{\bf Step $\ell$ ($\ell>2$):}
Clearly  $|W_{\ell}|\geq 5$ and $\|\mathcal{P}_{\ell}\|\geq 2.$ Set
\begin{eqnarray*}
\hspace{-0.7 cm}&&\overline{\mathcal{P}}_{\ell}=e_{2{\ell}-1}e_{2{\ell}},\\
\hspace{-0.7 cm}&&\mathcal{P}'_{\ell}=\Big((e_{2{\ell}-1}\setminus\{f_{\mathcal{P},e_{2{\ell}-1}}\})\cup
\{u_{{\ell}-1}\}\Big)e_{2{\ell}}.
\end{eqnarray*}
    Since $\mathcal{P}$
   is maximal w.r.t. $W_{\ell}$,
 using Lemma \ref{spacial configuration2} there is a good $\varpi_S$-configuration, say
 $\mathcal{Q}_{\ell}=f_{\ell}g_{\ell},$ in $\mathcal{H}_{\rm blue}$ with end vertices $x\in f_{\ell}$ and $y\in g_{\ell}$ in $W_{\ell}$
  such
 that   $\mathcal{Q}_{\ell}$ does not contain a vertex of  $e_{2{\ell}}\setminus e_{2{\ell}-1},$ say $u_{\ell}$.
 By Lemma \ref{spacial configuration2}, there are two subsets $W_{{\ell}1}\subseteq W_{\ell}$ and $W_{{\ell}2}\subseteq W_{\ell}$ with $|W_{{\ell}1}|\geq |W_{\ell}|-2$ and $|W_{{\ell}2}|\geq |W_{\ell}|-3$ so that for every distinct vertices $x'\in W_{{\ell}1}$ and $y'\in W_{{\ell}2}$, the path $\mathcal{Q}'_{\ell}=\Big((f_{\ell}\setminus\{x\})\cup\{x'\}\Big)\Big((g_{\ell}\setminus\{y\})\cup\{y'\}\Big)$ is also a good $\varpi_S$-configuration in $\mathcal{H}_{\rm blue}$ with  end vertices $x'$ and $y'$ in $W_{\ell}.$ Therefore, we may assume that
 either
  $\bigcup_{i=1}^{{\ell}}\mathcal{Q}_i$ is  a blue path $\mathcal{Q}$  with end
 vertices  in $W_{\ell}$  or we have  two disjoint blue paths $\mathcal{Q}$ and $\mathcal{Q}'$ with end vertices  in $W_{\ell}$   so that $\mathcal{Q}\cup \mathcal{Q}'=\bigcup_{i=1}^{{\ell}}\mathcal{Q}_i$.

\noindent Set $X_{\ell}=|W\setminus \bigcup_{i=1}^{{\ell}}V(\mathcal{Q}_i)|$ and
$\mathcal{P}_{{\ell}+1}=\mathcal{P}_{{\ell}}\setminus
\overline{\mathcal{P}}_{{\ell}}=e_{2{\ell}+1}e_{2{\ell}+2}\ldots e_n$. If
$\bigcup_{i=1}^{{\ell}}\mathcal{Q}_i$ is a blue path $\mathcal{Q}$ with end vertices $x_{{\ell}}$ and $y_{{\ell}}$,
then set
\begin{eqnarray*}
W_{{\ell}+1}=\Big(W_{{\ell}}\setminus
V(\mathcal{Q})\Big)\cup\{x_{{\ell}},y_{{\ell}}\}.
\end{eqnarray*}
Note that in this case, $|W_{{\ell}}|-2\leq |W_{{\ell}+1}|\leq |W_{{\ell}}|-1$.
Otherwise, $\bigcup_{i=1}^{{\ell}}\mathcal{Q}_i$ is the union of  two disjoint blue
paths $\mathcal{Q}$ and $\mathcal{Q}'$ with end vertices $x_{{\ell}},y_{{\ell}}$ and $x'_{{\ell}},y'_{{\ell}}$, respectively. In this case,  set
\begin{eqnarray*}
W_{{\ell}+1}=\Big(W_{{\ell}}\setminus
V(\mathcal{Q}\cup\mathcal{Q}')\Big)\cup\{x_{{\ell}},y_{{\ell}},x'_{{\ell}},y'_{{\ell}}\}.
\end{eqnarray*}
Clearly, $|W_{{\ell}}|-1\leq |W_{{\ell}+1}|\leq |W_{{\ell}}|$.\\
 If  $|W_{{\ell}+1}|\leq 4$ or $\|\mathcal{P}_{{\ell}+1}\|\leq 1$,
 then   $\bigcup_{i=1}^{{\ell}}\mathcal{Q}_i=\mathcal{Q}$ and  $\emptyset$  or $\mathcal{Q}$ and $\mathcal{Q}'$ (in the case $\bigcup_{i=1}^{{\ell}}\mathcal{Q}_i=\mathcal{Q}\cup\mathcal{Q}'$) are the paths with the desired properties. Otherwise, go to Step $\ell+1$.\\

Let $t\geq 2$ be the minimum integer for which we have either $|W_t|\leq
4$ or $\|\mathcal{P}_t\|\leq 1$.
 Set $x=X_{t-1}$ and $r=\|\mathcal{P}_{t}\|=n-2(t-1)$. So $\bigcup_{i=1}^{t-1}\mathcal{Q}_i$ is either  a blue path $\mathcal{Q}$ or the union two disjoint  blue paths $\mathcal{Q}$ and $\mathcal{Q}'$ between $\overline{\mathcal{P}}=e_1e_2\ldots
 e_{n-r}$ and $W'=W\cap (\bigcup_{i=1}^{t-1} V(\mathcal{Q}_i))$ with the desired properties.
 If $\bigcup _{i=1}^{t-1}\mathcal{Q}_i$ is  a blue path $\mathcal{Q},$ then either $x\in\{1,2\}$ or $x\geq
3$ and $0\leq r\leq 1$. Otherwise,
$\bigcup _{i=1}^{t-1}\mathcal{Q}_i$ is the union of  two disjoint blue paths
$\mathcal{Q}$ and $\mathcal{Q}'$ and we have either $x=0$ or $x\geq 1$ and $0\leq
r\leq 1$.

 }\end{proof}



\section{Ramsey number of 4-uniform loose cycles}

In this section we investigate Conjecture \ref{our conjecture2} for  $k=4$.
 Indeed, we determine the exact  value of
$R(\mathcal{C}^4_n,\mathcal{C}^4_m)$, where   $n> m\geq 3$ and $n=m$ is odd. When $n=m$ is even, we show that
 $R(\mathcal{C}^4_n,\mathcal{C}^4_n)$ is between two values with difference one.
 For this purpose we need the following essential  lemma.



\bigskip
\begin{lemma}\label{cn-1 implies cm for n>m}
Let $n\geq m\geq 3$, $(n,m)\neq (3,3),(4,3),(4,4)$  and

\begin{eqnarray*} t= \left\lbrace \begin{array}{ll}
\lfloor\frac{m-1}{2}\rfloor  &\mbox{ if\ \ $n>m$ },\vspace{.5 cm}\\
\lfloor\frac{m}{2}\rfloor  &\mbox{otherwise}.
\end{array} \right.\vspace{.2 cm} \end{eqnarray*}
Assume that  $\mathcal{H}=\mathcal{K}^4_{3n+t}$ is  $2$-edge
colored red and blue and  there is no copy of $\mathcal{C}^4_{n}$
in $\mathcal{H}_{\rm red}$. If
$\mathcal{C}=\mathcal{C}^4_{n-1}\subseteq \mathcal{H}_{\rm red}$,
then $\mathcal{C}^4_{m}\subseteq \mathcal{H}_{\rm blue}$.
\end{lemma}
\begin{proof}{Let
$\mathcal{C}=e_1e_2\ldots e_{n-1}$ be a copy of
$\mathcal{C}_{n-1}^4$ in $\mathcal{H}_{\rm red}$ with  edges
$$e_j=\{v_{3j-2},v_{3j-1},v_{3j},v_{3j+1}\}\hspace{0.5 cm}({\rm mod}\ \  3(n-1)),\hspace{1 cm} 1\leq j \leq n-1,$$ and $W=V(\mathcal{H})\setminus V(\mathcal{C})$. So we have $|W|=t+3.$ Consider
the following cases:

\bigskip
\noindent \textbf{Case 1. } For some edge $e_i=\{v_{3i-2},v_{3i-1},v_{3i},v_{3i+1}\},$ $1\leq i\leq n-1$, there is a
vertex $z\in W$ such that at least one of the edges  $e=\{v_{3i-1},v_{3i},v_{3i+1},z\}$
or $e'=\{v_{3i-2},v_{3i-1},v_{3i},z\}$ is red.

\medskip
We can clearly assume that the edge
$e=\{v_{3i-1},v_{3i},v_{3i+1},z\}$ is red. Set $$\mathcal{P}=e_{i+1}
e_{i+2}\ldots e_{n-1} e_1 e_2\ldots e_{i-2} e_{i-1}$$ and
$W_0=W\setminus \{z\}$ (If the edge $\{v_{3i-2},v_{3i-1},v_{3i},z\}$ is
red,  consider the path  $$\mathcal{P}=e_{i-1} e_{i-2}\ldots e_2 e_1
e_{n-1}\ldots e_{i+2} e_{i+1}$$ and  do the following process
to get a blue copy of $\mathcal{C}_m^4$).

 First let $m\leq 4$. Since $n\geq
5,$ we have $t=\lfloor\frac{m-1}{2}\rfloor=1$ and hence $|W_0|=3.$ Let
$W_0=\{u_1,u_2,u_3\}.$ We show that $\mathcal{H}_{\rm blue}$
contains $\mathcal{C}^4_m$ for each $m\in\{3,4\}$. Set $f_1=\{u_1,v_{3i-3},v_{3i-1},u_2\},$ $f_2=\{u_2,v_{3i-4},v_{3i},u_3\}$ and $f_3=\{u_3,z,v_{3i-2},u_1\}.$ Since there is no red copy of $\mathcal{C}_n^4,$ the edges $f_1,f_2$ and $f_3$  are blue. If not, let the edge $f_j,$ $1\leq j \leq 3,$ is red. Then $f_je e_{i+1}\ldots e_{n-1}e_1\ldots e_{i-1}$ is a red copy of $\mathcal{C}_n^4,$ a contradiction. So $f_1f_2f_3$ is a blue copy of  $\mathcal{C}_3^4.$
Also, since there is no red copy of $\mathcal{C}_n^4,$ the path
 $\mathcal{P}'=e_{i-3}e_{i-2}$ (we use mod $(n-1)$ arithmetic) is
maximal w.r.t. $W=W_0\cup\{z\}$. Using Lemma \ref{spacial
configuration2}, there is a good $\varpi_S$-configuration, say $C=fg,$ in $\mathcal{H}_{\rm blue}$ with end vertices $x\in f$ and  $y\in g$ in $W$ and $S\subseteq e_{i-3}e_{i-2}$. Note that, by Lemma \ref{spacial configuration2}, there are two subsets $W_1$ and $W_2$ of $W$ with $|W_1|\geq 2$ and $|W_2|\geq 1$ so that
for every distinct vertices $x'\in W_1$ and $y'\in W_2$, the path $C'=\Big((f\setminus\{x\})\cup\{x'\}\Big)\Big((g\setminus\{y\})\cup\{y'\}\Big)$ is also a good $\varpi_S$-configuration in $\mathcal{H}_{\rm blue}$ with  end vertices $x'$ and $y'$ in $W.$ Clearly,  at least one of the vertices  of $W_0,$ say $u_1,$
 is an end vertex of $C$. Let  $u\in
W_0\setminus V(C)$. Set $g_1= \{u_2,u_3,z,v_{3i-2}\}$ and $g_2=\{u,v_{3i-3},v_{3i-1},u_1\}.$ Since the edge $e$ is red, the edges $g_1$ and $g_2$ are blue (otherwise, we can find a red copy of $\mathcal{C}_n^4$) and $Cg_1g_2$
 is a blue copy of $\mathcal{C}_4^4$.\\

Now let $m\geq 5.$ Clearly $|W_0|=t+2\geq 4$. Since there is no red
copy of $\mathcal{C}^4_n$, $\mathcal{P}$ is a maximal path w.r.t.
$W_0$. Applying Corollary \ref{there is a Pl}, there are  two
disjoint blue paths $\mathcal{Q}$ and $\mathcal{Q}'$
 between
 $\overline{\mathcal{P}}$, the path obtained from $\mathcal{P}$ by deleting the last $r$ edges for some $r\geq 0$,
  and $W'\subseteq W_0$ with the mentioned properties. Consider the paths
 $\mathcal{Q}$ and $\mathcal{Q}'$ with $\|\mathcal{Q}\|\geq \|\mathcal{Q}'\|$
  so that $\ell'=\|\mathcal{Q}\cup\mathcal{Q}'\|$ is maximum. Among these paths choose $\mathcal{Q}$ and $\mathcal{Q}'$, where $\|\mathcal{Q}\|$ is maximum.
  Since $\|\mathcal{P}\|=n-2, $
  by Corollary \ref{there is a Pl}, we have $r=n-2-\ell'$.\\


\noindent {\it Subcase $1$}.
 $\|\mathcal{Q}'\|\neq 0$.\\
  Set $T=W_0\setminus W'.$ Let  $x,y$ and $x',y'$ be the end vertices of $\mathcal{Q}$
and $\mathcal{Q}'$ in $W'$, respectively. Using Corollary
\ref{there is a Pl}, we have one of the following cases:

\begin{itemize}
\item[I.] $|T|\geq 2.$\\
It is easy to see that $\ell'\leq 2t-4$ and so
 $r\geq 2$. Hence this case does not occur by
Corollary \ref{there is a Pl}.

\item[II.] $|T|=1.$\\
Let
$T=\{u\}$.
One can easily check that $\ell'=2t-2.$  If $n>m$, then $r\geq 2,$  a contradiction to Corollary \ref{there is a Pl}.
Therefore, we may assume that $n=m$.
 If $n$ is even, then  $\ell'=n-2$. Remove the last two edges of
$\mathcal{Q}\cup \mathcal{Q}'$ to get two disjoint blue paths
$\overline{\mathcal{Q}}$ and $\overline{\mathcal{Q}'}$ so that  $\|\overline{\mathcal{Q}}\cup \overline{\mathcal{Q}'}\|=n-4$ and  $(\overline{\mathcal{Q}}\cup \overline{\mathcal{Q}'})\cap((e_{i-2}\setminus\{f_{\mathcal{P},e_{i-2}}\})\cup e_{i-1})=\emptyset$ (note that by the proof of Corollary \ref{there is a Pl}, this is possible). By Corollary \ref{there is a Pl}, there is a vertex
 $w\in e_{i-3}\setminus e_{i-4}$ so that $w\notin
V(\overline{\mathcal{Q}}\cup \overline{\mathcal{Q}'}).$
 We can without loss of generality assume that $\mathcal{Q}=\overline{\mathcal{Q}}.$ First
let $\|\overline{\mathcal{Q}'}\|>0$ and $x',y''$ with  $y''\neq y'$ be
end vertices of $\overline{\mathcal{Q}'}$ in $W'.$ Set
\begin{eqnarray*}
f_1=\{y'',v_{3i-3},v_{3i-1},u\}, f_2=\{u,z,v_{3i-2},y'\},  f_3=\{y',v_{3i},v_{3i-4},x\}.
\end{eqnarray*}
Since the edge $e$ is red, then the edges $f_i,$ $1\leq i \leq 3,$ are blue (otherwise we can find a red copy of $\mathcal{C}_n^4$, a contradiction to our assumption).
 If the edge
$f=\{y,w,v_{3i-7},x'\}$ is blue, then $\mathcal{Q}f\overline{\mathcal{Q}'}f_1f_2f_3$
 is a copy of
$\mathcal{C}^4_m$ in $\mathcal{H}_{\rm blue}$.
Otherwise, the edge
$g=\{y,v_{3i-6},v_{3i-5},y''\}$ is blue (if not, $fge_{i-1} \ldots e_{n-1} e_1 \ldots e_{i-3}$ is a red copy of $\mathcal{C}_n^4,$ a contradiction).
 Also, since there is no red copy of $\mathcal{C}_n^4,$ the edges
\begin{eqnarray*}
g_1=\{x',v_{3i-3},v_{3i-1},u\}, g_2=\{u,z,v_{3i-2},y'\}, g_3=\{y',v_{3i},v_{3i-4},x\},
\end{eqnarray*}
are blue. Clearly  $\mathcal{Q}g\overline{\mathcal{Q}'}g_1g_2g_3$
 is a
blue copy of $\mathcal{C}^4_m$. Now, we may assume that
$\|\overline{\mathcal{Q}'}\|=0.$   In this case, set
$f'=\{y,w,v_{3i-7},x'\}$. If the edge  $f'$ is blue, then $\mathcal{Q}f'g_1g_2g_3$
  is a blue copy of $\mathcal{C}^4_m$. Otherwise,
 the edge  $g'=\{y,v_{3i-6},v_{3i-5},y'\}$ is blue (if not, $f'g'e_{i-1} \ldots e_{n-1} e_1 \ldots e_{i-3}$ makes a red $\mathcal{C}_n^4$). Similarly, since there is no red copy of $\mathcal{C}_n^4$ and the edge $e$ is red,
\begin{eqnarray*}
\mathcal{Q}g'\{y',v_{3i-3},v_{3i-1},u\}\{u,z,v_{3i-2},x'\}\{x',v_{3i},v_{3i-4},x\},
\end{eqnarray*}
is a blue copy of  $\mathcal{C}^4_m$. \\

Therefore, we may assume that
 $n$ is odd. Clearly,  $\ell'=n-3$ and $r\geq 1$. Again, since there is no red copy of $\mathcal{C}_n^4,$ the edges
\begin{eqnarray*}
h_1=\{y,v_{3i-4},v_{3i-1},x'\}, h_2=\{y',v_{3i-2},z,u\}, h_3=\{u,v_{3i},v_{3i-3},x\},
\end{eqnarray*}
 are blue and $\mathcal{Q}h_1\mathcal{Q}'h_2h_3,$
 makes a copy of $\mathcal{C}^4_m$ in $\mathcal{H}_{\rm blue}$.

\item[III.] $|T|=0.$\\
Clearly we have $\ell'=2t$. First let $m$ be odd. Therefore, we have  $\ell'=m-1.$
Remove the last two edges of $\mathcal{Q}\cup \mathcal{Q}'$ to get
two disjoint blue paths $\overline{\mathcal{Q}}$ and
$\overline{\mathcal{Q}'}$ so that $\|\overline{\mathcal{Q}}\cup \overline{\mathcal{Q}'}\|=m-3$ and  $(\overline{\mathcal{Q}}\cup \overline{\mathcal{Q}'})\cap((e_{i-2}\setminus\{f_{\mathcal{P},e_{i-2}}\})\cup e_{i-1})=\emptyset$ (this is possible, by the proof of Corollary \ref{there is a Pl}). We can without loss of generality assume that
$\mathcal{Q}=\overline{\mathcal{Q}}.$ First
let $\|\overline{\mathcal{Q}'}\|>0$ and $x',y''$ with  $y''\neq y'$ be
end vertices of $\overline{\mathcal{Q}'}$ in $W'.$ Since the edge $e$ is red and  there is no red copy of $\mathcal{C}_n^4,$ the edges
\begin{eqnarray*}
f_1=\{y,v_{3i-3},v_{3i-1},x'\}, f_2=\{y'',v_{3i-4},v_{3i},y'\}, f_3=\{y',z,v_{3i-2},x\},
\end{eqnarray*}
are blue and so $\mathcal{Q}f_1\overline{\mathcal{Q}'}f_2f_3$
is a blue copy of $\mathcal{C}^4_m$.
Now let
$\|\overline{\mathcal{Q}'}\|=0$. Again, since there is no red copy of $\mathcal{C}_n^4,$ the edge $g_1=\{x',v_{3i-4},v_{3i},y'\}$
 is  blue and
$\mathcal{Q}f_1g_1f_3,$
 is a blue  copy of $\mathcal{C}^4_m$.

Now let $m$ be even. If $n>m,$ then
 $\ell'=m-2$ and
$r\geq 1$. Clearly,
\begin{eqnarray*}
\mathcal{Q}\{y,v_{3i-3},v_{3i-1},x'\}\mathcal{Q}'\{y',v_{3i},v_{3i-4},x\},
\end{eqnarray*}
 is a blue copy of  $\mathcal{C}_m^4$. Therefore, we may assume that $n=m$. Thereby  $\ell'=m$.
Remove the last two edges of
$\mathcal{Q}\cup \mathcal{Q}'$ to get two disjoint blue paths
$\overline{\mathcal{Q}}$ and $\overline{\mathcal{Q}'}$ so that $\|\overline{\mathcal{Q}}\cup \overline{\mathcal{Q}'}\|=m-2$ and
 $(\overline{\mathcal{Q}}\cup \overline{\mathcal{Q}'})\cap((e_{i-2}\setminus\{f_{\mathcal{P},e_{i-2}}\})\cup e_{i-1})=\emptyset$.
 We can without loss of generality assume that $\mathcal{Q}=\overline{\mathcal{Q}}.$ First
let $\|\overline{\mathcal{Q}'}\|>0$ and $x',y''$ with  $y''\neq y'$ be
end vertices of $\overline{\mathcal{Q}'}$ in $W'.$ Since there is no red copy of $\mathcal{C}^4_n,$ the edges $h_1=\{y,v_{3i-3},v_{3i-1},x'\}$ and $h_2=\{y'',v_{3i},v_{3i-4},x\}$ are blue and $\mathcal{Q}h_1\overline{\mathcal{Q}'}h_2$ forms a blue copy of $\mathcal{C}^4_m.$ If $\|\overline{\mathcal{Q}'}\|=0,$ then $\mathcal{Q}h_1\{x',v_{3i},v_{3i-4},x\}$ is a blue copy of $\mathcal{C}^4_m.$

\end{itemize}

\noindent {\it Subcase $2$}. $\|\mathcal{Q}'\|=0$.\\
  Let $x$ and $y$  be the end
vertices of $\mathcal{Q}$ in $W'$ and  $T=W_0\setminus W'$. Using Corollary \ref{there is a Pl} we have the following:

\begin{itemize}
\item[I.] $|T|\geq 3$.\\
In this case, clearly $\ell'\leq 2(t-2)$ and so
 $r\geq 2$. This is a contradiction to Corollary \ref{there is a Pl}.

\item[II.] $|T|=2.$\\
Let $T=\{u_1,u_2\}$. So we have   $\ell'=2t-2$. First let $m$ be odd.  Hence,
$\ell'=m-3$ and $r\geq 1$. Since there is no red copy of $\mathcal{C}_n^4$ and the edge $e$ is red,
 the edges
 \begin{eqnarray*}
f_1=\{y,v_{3i-4},v_{3i-1},u_1\}, f_2=\{u_1,v_{3i-3},v_{3i},u_2\}, f_3=\{u_2,z,v_{3i-2},x\},
\end{eqnarray*}
are blue. If not, suppose that the edge $f_j,$ $1\leq j\leq 3,$ is red. So  $f_j e e_{i+1}e_{i+2} \ldots$ $ e_{n-1} e_1 \ldots e_{i-1}$ is a red copy of $\mathcal{C}^4_n,$ a contradiction. Thereby,
 $\mathcal{Q}f_1f_2f_3$
makes a blue copy of $\mathcal{C}^4_m$.

Now let $m$
be even. If $n>m,$ then  $\ell'=m-4$ and $r\geq 3$. Using Corollary
\ref{there is a Pl}, there is a vertex $w\in e_{i-4}\setminus
e_{i-5}$ so that $w\notin V(\mathcal{Q}).$
Since $\mathcal{P}$ is maximal w.r.t. $\overline{W}=\{x,y,u_1,u_2,z\}$,
using Lemma \ref{spacial
configuration2},  there is a good $\varpi_S$-configuration,
say $C_1=fg$, in $\mathcal{H}_{\rm blue}$ with end vertices $x'\in f$ and $y'\in g$ in $\overline{W}$ and
\begin{eqnarray*}
S\subseteq \Big((e_{i-3}\setminus f_{\mathcal{P},e_{i-3}})\cup \{w\}\Big)\cup e_{i-2}.
\end{eqnarray*}
Moreover, by Lemma \ref{spacial configuration2}, there are two subsets $W_1$ and $W_2$ of $\overline{W}$ with $|W_1|\geq 3$ and $|W_2|\geq 2$ so that
for every distinct vertices $\overline{x'}\in W_1$ and $\overline{y'}\in W_2$, the path $C'_1=\Big((f\setminus\{x'\})\cup\{\overline{x'}\}\Big)\Big((g\setminus\{y'\})\cup\{\overline{y'}\}\Big)$ is also a good $\varpi_S$-configuration in $\mathcal{H}_{\rm blue}$ with  end vertices $\overline{x'}$ and $\overline{y'}$ in $\overline{W}.$
 Since $|W_1|\geq 3$ and $\ell'$ is maximum,
 we may
assume that $y$ and $z$ or $x$ and $z$ are  end vertices of $C_1$ in $\overline{W}$. By symmetry suppose that $y$ and $z$  are  end vertices of $C_1$ in $\overline{W}.$ Since there is no red copy of $\mathcal{C}^4_n$ and the edge $e$ is red, then
\begin{eqnarray*}
\mathcal{Q}C_1\{z,v_{3i-2},u_1,u_2\}\{u_2,v_{3i-1},v_{3i-3},x\},
\end{eqnarray*}
is a blue copy of $\mathcal{C}^4_m$. Now, we may assume that $n=m$. Clearly $\ell'=m-2.$ By Corollary \ref{there is a Pl}, there is a vertex $w'\in e_{i-1}\setminus
e_{i-2}$ so that $w'\notin V(\mathcal{Q}).$ Again, since  there is no copy of
 $\mathcal{C}^4_n$ in $\mathcal{H}_{\rm red},$ so
\begin{eqnarray*}
\mathcal{Q}\{y,u_1,v_{3i-1},w'\}\{w',v_{3i},u_2,x\},
\end{eqnarray*}
 is a blue copy
of $\mathcal{C}^4_m$.

\item[III.] $|T|=1.$\\
Clearly  $\ell'=2t$. Let $T=\{u_1\}$.  First let $m$ be odd. Therefore, $\ell'=m-1$.
By Corollary \ref{there is a Pl} there is a vertex
 $w\in e_{i-1}\setminus e_{i-2}$ so that
$w\notin V(\mathcal{Q})$. Clearly the edge $g=\{y,w,z,x\}$ is blue (otherwise $gee_{i+1}\ldots e_{n-1}e_1\ldots e_{i-1}$ makes a red copy of $\mathcal{C}^4_n$). Thereby $\mathcal{Q}g$
 is a blue $\mathcal{C}^4_m$.
Now, suppose  that $m$ is even. If $n>m$, then $\ell'=m-2$ and
$r\geq 1$.  Since the edge $e$ is red and  there is no red copy of $\mathcal{C}^4_{n}$,
\begin{eqnarray*}
\mathcal{Q}\{y,v_{3i-2},z,u_1\}\{u_1,v_{3i-1},v_{3i-3},x\},
\end{eqnarray*}
 is a copy of
$\mathcal{C}^4_m$ in $\mathcal{H}_{\rm blue}$. If $n=m$, then
 $\ell'=m$.
In this case, remove the last two edges of
$\mathcal{Q}$ to get two disjoint blue paths
$\overline{\mathcal{Q}}$ and $\overline{\mathcal{Q}'}$ so that $\|\overline{\mathcal{Q}}\cup \overline{\mathcal{Q}'}\|=m-2$ and $(\overline{\mathcal{Q}}\cup \overline{\mathcal{Q}'})\cap((e_{i-2}\setminus\{f_{\mathcal{P},e_{i-2}}\})\cup e_{i-1})=\emptyset$. By symmetry we may assume that $\|\overline{\mathcal{Q}}\|\geq \|\overline{\mathcal{Q}'}\|. $
  First
suppose that  $\|\overline{\mathcal{Q}'}\|=0.$ Then we may suppose that $x,y'$ with  $y'\neq y$ be
end vertices of $\overline{\mathcal{Q}}$ in $W'.$ Since there is no red copy of $\mathcal{C}^4_n$ and the edge $e$ is red, the edges
 $h_1=\{y',v_{3i-3},v_{3i-1},y\}$ and $h_2=\{y,v_{3i},v_{3i-4},x\}$ are blue and $\overline{\mathcal{Q}}h_1h_2$ forms a blue copy of $\mathcal{C}^4_m.$ So we may assume that  $\|\overline{\mathcal{Q}'}\|>0$. Let $x',y'$ and $x'',y''$ be end vertices of $\overline{\mathcal{Q}}$ and $\overline{\mathcal{Q}'}$ in $W',$ respectively. One can easily check that
 \begin{eqnarray*}
  \overline{\mathcal{Q}}\{y',v_{3i-3},v_{3i-1},x''\}\overline{\mathcal{Q}'}\{y'',v_{3i},v_{3i-4},x'\},
  \end{eqnarray*}
 is a blue copy of $\mathcal{C}^4_m.$


\item[IV.] $|T|=0.$\\
Clearly, we have  $\ell'=2t+2$.  First let $m$ be odd. Therefore, $\ell'=m+1$.
 Remove the last two edges of
$\mathcal{Q}$ to get two disjoint blue paths
$\overline{\mathcal{Q}}$ and $\overline{\mathcal{Q}'}$ so that $\|\overline{\mathcal{Q}}\cup \overline{\mathcal{Q}'}\|=m-1$ and $(\overline{\mathcal{Q}}\cup \overline{\mathcal{Q}'})\cap((e_{i-2}\setminus\{f_{\mathcal{P},e_{i-2}}\})\cup e_{i-1})=\emptyset$. By symmetry we may assume that $\|\overline{\mathcal{Q}}\|\geq \|\overline{\mathcal{Q}'}\|. $ If $\|\overline{\mathcal{Q}'}\|=0,$ then we may suppose that $x,y'$ with  $y'\neq y$ be
end vertices of $\overline{\mathcal{Q}}$ in $W'.$
Clearly the edge $g=\{y',v_{3i-2},z,x\}$ is blue (otherwise $gee_{i+1}\ldots e_{n-1}e_1\ldots e_{i-1}$ makes a red copy of $\mathcal{C}^4_n$). Thereby $\overline{\mathcal{Q}}g$
 is a blue $\mathcal{C}^4_m$. If $\|\overline{\mathcal{Q}'}\|>0,$ then remove the last two edges of
$\overline{\mathcal{Q}}\cup \overline{\mathcal{Q}'}.$
By an argument similar to the case $\|\mathcal{Q}'\|\neq 0$ and $|T|=0$,
 we can find a blue copy of $\mathcal{C}^4_m.$ When $m$ is even, by removing the last two edges of $\mathcal{Q}$, one of the before cases holds. So
  we omit the proof here.

\end{itemize}
\bigskip
\noindent \textbf{Case 2. } For every edge
$e_i=\{v_{3i-2},v_{3i-1},v_{3i},v_{3i+1}\}$, $1\leq i\leq n-1$,
and every vertex $z\in W$,  the edges
$\{v_{3i-1},v_{3i},v_{3i+1},z\}$ and
$\{v_{3i-2},v_{3i-1},v_{3i},z\}$ are blue.

\medskip
Let $W=\{x_1,x_2,\ldots,x_{t},u_1,u_2,u_3\}$. We have two following  subcases:\\

\noindent {\it Subcase 1.}
 For some edge
$e_j=\{v_{3j-2},v_{3j-1},v_{3j},v_{3j+1}\}$, $1\leq j\leq n-1$,
there are
 vertices $u$ and
$v$ in $W$ so that at least one of the  edges
$\{v_{3j-2},v_{3j-1},u,v\}$ or $\{v_{3j},v_{3j+1},u,v\}$ is
blue.\\
We can without loss of generality   assume that the edge
$\{v_{3j-2},v_{3j-1},u,v\}$ is blue (if the edge
$\{v_{3j},v_{3j+1},u,v\}$ is blue, the proof is similar). By
symmetry we may assume that $e_j=e_1$ and $\{u,v\}=\{u_1,u_2\}$.
Set
\begin{eqnarray*}
&&e_{0}^{\prime}=(e_{1}\setminus\{v_{3},v_4\})\cup\{u_{1},u_2\},\\
&&e_{1}^{\prime}=(e_{1}\setminus\{v_{1}\})\cup\{x_1\}
\end{eqnarray*}
For $2\leq i\leq m-2$ set
\begin{eqnarray*}
e_i^{\prime}= \left\lbrace
\begin{array}{ll}
(e_i\setminus \{l_{\mathcal{C},e_i}\})\cup\{x_{\frac{i+1}{2}}\}  &\mbox{if~}  i~\mbox{is~odd},\vspace{.5 cm}\\
(e_i\setminus\{f_{\mathcal{C},e_i}\})\cup\{x_{\frac{i}{2}}\}
&\mbox{if~} i~\mbox{is~even}.
\end{array}
\right.\vspace{.2 cm}
\end{eqnarray*}
\noindent Also, let
\begin{eqnarray*}
 e'_{m-1}= \left\lbrace
\begin{array}{ll}
 (e_{m-1}\setminus
\{l_{\mathcal{C},e_{m-1}}\})\cup \{u_1\}  &\mbox{if $m$ is
even},\vspace{.5 cm}\\ (e_{n-1}\setminus
\{f_{\mathcal{C},e_{n-1}}\})\cup \{x_{\frac{m-1}{2}}\} &\mbox{if
$m$ is odd}. \end{array} \right.\vspace{.2 cm}
\end{eqnarray*}
 Thereby,
$e_0^{\prime}e_1^{\prime}\ldots e^{\prime}_{m-1}$ forms a blue copy of
$\mathcal{C}_m^4$.\\

\noindent {\it Subcase 2.} For every edge
$e_j=\{v_{3j-2},v_{3j-1},v_{3j},v_{3j+1}\}$, $1\leq j\leq n-1$,
and every  vertices $u,v$  in $W$, the edges
$\{v_{3j-2},v_{3j-1},u,v\}$ and $\{v_{3j},v_{3j+1},u,v\}$ are
red.\\
One can easily check that
\begin{eqnarray*}
\{v_1,v_2,u_1,u_2\}\{u_2,u_3,v_3,v_4\}e_2\ldots e_{n-1},
\end{eqnarray*}
 is a red copy of  $\mathcal{C}_n^4$. This contradiction completes the proof.

 }\end{proof}




 The following results are the main results of this
section.


 \bigskip
\begin{theorem}\label{main theorem4}
For every $n\geq m+1\geq 4$,
$$R(\mathcal{C}^4_n,\mathcal{C}^4_m)=3n+\Big\lfloor\frac{m-1}{2}\Big\rfloor.$$
\end{theorem}

\begin{proof} {We give a proof by induction on $m+n$.  By  Theorems
\ref{R(Pk3,Pk3)} and  \ref{R(C3,C4)} we may assume that $n\geq 5$.
Suppose to the contrary that
$\mathcal{H}=\mathcal{K}^4_{3n+\lfloor\frac{m-1}{2}\rfloor}$ is
2-edge colored red and blue with  no  red copy of
$\mathcal{C}^4_n$  and no  blue copy of $\mathcal{C}^4_m$.
 Consider the following cases:

\bigskip
\noindent \textbf{Case 1. } $n=m+1.$

\medskip
By induction hypothesis,
\begin{eqnarray*}
R(\mathcal{C}^4_{n-1},\mathcal{C}^4_{n-2})=
3(n-1)+\Big\lfloor\frac{n-3}{2}\Big\rfloor< 3n+
\Big\lfloor\frac{n-2}{2}\Big\rfloor.
\end{eqnarray*}

\noindent If there is a  copy of $\mathcal{C}^4_{n-1}$ in
$\mathcal{H}_{\rm red},$ then using Lemma \ref{cn-1 implies cm for
n>m} we have a blue copy of $\mathcal{C}^4_{n-1}.$  So we may
assume that there is no red copy of $\mathcal{C}^4_{n-1}.$
Therefore, there is a copy of $\mathcal{C}^4_{n-2}$ in
$\mathcal{H}_{\rm blue}.$ Since there is no blue copy of $\mathcal{C}^4_{n-1},$
 applying Lemma \ref{cn-1 implies cm for
n>m},   we have a red copy of $\mathcal{C}^4_{n-1}.$ This  is a contradiction
to our assumption.

\bigskip
\noindent \textbf{Case 2. }$n> m+1$.

\medskip

By the induction hypothesis
\begin{eqnarray*} R(\mathcal{C}^4_{n-1},\mathcal{C}^4_{m})=
3(n-1)+\Big\lfloor\frac{m-1}{2}\Big\rfloor<
3n+\Big\lfloor\frac{m-1}{2}\Big\rfloor.
\end{eqnarray*}
 Since there is no blue
copy of $\mathcal{C}^4_m$, we have a copy of $\mathcal{C}^4_{n-1}$
in $\mathcal{H}_{\rm red}$.
  Using  Lemma \ref{cn-1 implies cm for n>m}, we have a blue copy of $\mathcal{C}^4_{m}$. This contradiction completes the proof.
}\end{proof}

\bigskip
\begin{theorem}\label{main theorem1}
For every $n\geq  4$,
\begin{eqnarray*}R(\mathcal{C}^4_n,\mathcal{C}^4_n)\leq 3n+\Big\lfloor\frac{n}{2}\Big\rfloor.
\end{eqnarray*}
\end{theorem}
\begin{proof} { We give a proof  by induction on $n$. Applying Theorem
\ref{R(Pk3,Pk3)} the statement is true for $n=4$.
 Suppose  that, on the contrary,  the edges of
 $\mathcal{H}=\mathcal{K}^3_{3n+\lfloor\frac{n}{2}\rfloor}$ can be
 colored red and blue with  no  red copy of
$\mathcal{C}^4_n$  and no  blue copy of $\mathcal{C}^4_n$. By the
induction assumption,
\begin{eqnarray*}R(\mathcal{C}^4_{n-1},\mathcal{C}^4_{n-1})\leq
3(n-1)+\Big\lfloor\frac{n-1}{2}\Big\rfloor<
3n+\Big\lfloor\frac{n}{2}\Big\rfloor.
\end{eqnarray*} By symmetry  we may assume that there is a red copy of $\mathcal{C}^4_{n-1}$.
Using Lemma \ref{cn-1 implies cm for n>m} we have a  copy of
$\mathcal{C}^4_{n}$ in $\mathcal{H}_{\rm blue}$. This is a
contradiction. }\end{proof}

Using Lemma 1 of \cite{subm} and Theorem \ref{main theorem1} we
conclude the following corollary.

\begin{corollary}\label{R(C^4_n,C^4_n)}
Let  $n\geq 4$.  If $n$ is odd, then
$R(\mathcal{C}^4_n,\mathcal{C}^4_n)=
3n+\Big\lfloor\frac{n-1}{2}\Big\rfloor.$ Otherwise,
\begin{eqnarray*}3n+\Big\lfloor\frac{n-1}{2}\Big\rfloor \leq
R(\mathcal{C}^4_n,\mathcal{C}^4_n)\leq
3n+\Big\lfloor\frac{n-1}{2}\Big\rfloor+1.
\end{eqnarray*}
\end{corollary}

Clearly using the above results on the Ramsey number of loose cycles and Theorem \ref{connection}, we obtain the following results.

\begin{theorem}
If $n\geq m+1\geq 4$ or $n=m$ is odd, then
\begin{eqnarray*} R(\mathcal{P}^4_n,\mathcal{C}^4_m)=3n+\Big\lfloor\frac{m+1}{2}\Big\rfloor.\end{eqnarray*}
\end{theorem}

\begin{theorem} Let $n\geq m\geq 3.$ 
If $n\geq m+2\geq 5$ or $n$ is odd, then
\begin{eqnarray*}R(\mathcal{P}^4_n,\mathcal{P}^4_m)=3n+\Big\lfloor\frac{m+1}{2}\Big\rfloor.\end{eqnarray*}
\end{theorem}

\end{document}